\documentclass[a4paper,12pt,twoside]{article}

\usepackage{amsmath,amsthm,amscd,amsfonts,amssymb,amsxtra}
\usepackage[hmargin=1.2in,vmargin=1.2in]{geometry}
\usepackage[utf8]{inputenc}
\usepackage{graphicx}
\usepackage{microtype}
\usepackage[pdftex,bookmarks,colorlinks=false]{hyperref}
\usepackage{booktabs,enumitem,caption,subcaption}
\usepackage[mathscr]{euscript}

\usepackage[auth-sc-lg,affil-it]{authblk}
\usepackage{setspace}
\numberwithin{equation}{section}

\newtheorem{thm}{Theorem}[section]
\newtheorem{defn}{Definition}[section]
\newtheorem{lem}[thm]{Lemma}
\newtheorem{prop}[thm]{Proposition}

\newtheorem{exa}{Example}[section]

\def\ni{\noindent}
\def\N{\mathbb{N}}
\def\A{\mathbb{A}}
\def\cP{\mathcal{P}}

\title{\textbf{\sc Tattooing and the Tattoo Number of Graphs}}

\author{Johan Kok}
\affil{\small Tshwane Metropolitan Police Department\\ City of Tshwane, Republic of South Africa \\ {\tt kokkiek2@tshwane.gov.za.}}

\author{Naduvath Sudev}
\affil{\small Department of Mathematics\\ Vidya Academy of Science \& Technology \\ Thalakkottukara, Thrissur - 680501, India.\\ {\tt sudevnk@gmail.com}}

\date{}

\begin{document}
\maketitle

\begin{abstract}
Consider a network $D$ of pipes which have to be cleaned using some cleaning agents, called brushes, assigned to some vertices. The minimum number of brushes required for cleaning the network $D$ is called its brush number. The tattooing of a simple connected directed graph $D$ is a particular type of the cleaning in which an arc are coloured by the colour of the colour-brush transiting it and the tattoo number of $D$ is a corresponding derivative of brush numbers in it. Tattooing along an out-arc of a vertex $v$ may proceed if a minimum set of colour-brushes is allocated (primary colours) or combined with those which have arrived (including colour blends) together with mutation of permissible new colour blends, has cardinality greater than or equal to $d^+_G(v)$. 
\end{abstract}


\vspace{0.35cm}

\ni \textbf{Mathematics Subject Classification:} 05C20, 05C35, 05C38, 05C99.
 
\section{Introduction}

For a general reference to notation and concepts of graph theory and digraph theory, not defined specifically in this paper, please see \cite{BM1,BLS,CL1,FH,EWW,DBW}.  For the graph colouring concepts, please see \cite{CZ1,JT1}. Unless mentioned otherwise, all graphs and digraphs considered in this paper are simple, finite, connected and non-trivial.

As a combination of graph searching problems (see \cite{TDP1,TDP2}) and chip firing problems (see \cite{BLS1}), a graph cleaning model was introduced in \cite{SM1}, which can be explained as follows. Consider a network $D$ of pipes that have to be periodically cleaned. For this purpose, we can use some cleaning agents called \textit{brushes}, assigned to some vertices of $D$. When vertex is cleaned, a brush must travel down each contaminated edge. When a brush traverse an edge, that edge is said to be cleaned and a graph $G$ is considered to be cleaned when every edge of $G$ has been cleaned. 

Some important and interesting studies on the brush number $b_r(G)$ of certain graphs have been done (For some of these studies, see \cite{KSK1,MEM1,MEM2}). In the cleaning models, the problem is initially set in such a way that all edges of a simple connected undirected graph $G$ are dirty. A finite number of brushes, $\beta_G(v) \ge 0$ is allocated to each vertex $v \in V(G)$. Sequentially, any vertex $v$, which has $\beta_G(v)$ brushes allocated to it, may clean the vertex itself and send exactly one brush along a dirty edge and in doing so allocate an additional brush to the corresponding adjacent vertex (neighbour). The reduced graph $G'=G-vu,\ \forall\, vu \in E(G),\ \beta_G(v) \ge d(v)$ is considered for the next iterative cleansing step. It is to be noted that a neighbour of the vertex $v$ in $G$, say $u$, now have $\beta_{G'}(u) =\beta_G(u) + 1$ brushes.

Clearly, for any simple connected undirected graph $G$ the first step of cleaning begin if and only if the number of brushes allocated to a vertex $v$ is greater than or equal to the degree of $v$; that is, $\beta_G(v)= d(v)$. Then, the minimum number of brushes that is required to commence the first step of cleaning is $\beta_G(u)=d(u)=\delta(G);\ u\in V(G)$. Also, note that the above mentioned condition does not guarantee that the graph will be cleaned, but just assures at least the occurrence of the first step of cleaning only.

If some orientation of the edges of a simple connected graph $G$ is given, then $G$ is a directed graph and brushes may only clean along an out-arc from a vertex. Cleaning may initiate from a vertex $v$ if and only if $\beta_G(v) \ge d_G^+(v)$ and $d_G^-(v) = 0$. The sequence in which vertices sequentially initiate cleaning is called the \textit{cleaning sequence} with respect to the orientation $\alpha_i(G)$. The minimum number of brushes to be allocated to clean a graph for a given orientation $\alpha_i(G)$ is denoted $b_r^{\alpha_i}(G)$. If an orientation $\alpha_i(G)$ renders cleaning of the graph unrealisable, define $b_r^{\alpha_i}(G)=\infty$. 

An orientation $\alpha_i(G)$ for which $b_r^{\alpha_i}(G)$ is a minimum over all possible orientations of $G$ is called an \textit{optimal orientation} of $G$ with respect to that cleaning model. This minimum number of brushes is called the \textit{brush number} of the graph $G$, denoted by $b_r(G)$. It is conventional to use $b_r(\mathcal{N}_n) = n$, where $\mathcal{N}_n$ is the null graph (edgeless graph) of order $n$.

If $\epsilon$ denotes the size of the graph $G$, then $G$ can have $2^{\epsilon}$ orientations and hence an \textit{optimal orientation} of $G$ need not be a unique one. Hence, we can define the collection of orientations of $G$, denoted by $\A$, can be defined as $\A =\{\alpha_i(G): \alpha_i\ \text{is an orientation of}\ G\}$.

Motivated by various studies on brush number on graphs, in this paper, we now introduce certain new colouring parameters derived from the brush number of graphs and study some important and interesting properties of these parameters.  

\section{Tattooing of Graphs}

In all studies on brush number of graphs, brushes are considered to be identical objects and when cleaning is initiated from a vertex $v$, a particular brush can clean along any out-arc at random. There is no preferred condition set on the choice of an out-arc. In the following study of tattooing of graphs, we use the property of colouring. The \textit{tattooing} of a simple connected directed graph $D$ is a particular type of the cleaning in which an arc will be coloured by the colour of the colour-brush transiting it and the tattoo number of the directed graph $D$ is a corresponding derivative of brush numbers in it. Initial colours will be called \textit{primary colours} and primary colours will be allowed to blend into additional \textit{colour blends}. 

The motivation for this mathematical model lies in various observations that many dynamical systems propagate through the system (considered to be a network or graph) upon reaching a threshold value. Certain other examples are electrical potential in cloud-earth systems to allow electrical arcing to create lightning and the associated sound of thunder. Some viruses or anti-viruses be it, biological or virtual, have a period of incubation to allow for either growth in numbers or mutation before propagation. It is common that pressure systems function on the diffusion of peak pressures through pressure valves. Perhaps the most historical example is the creation of our universe with the first event, called the Big Bang.

Consider a set of colour-brushes $\mathcal{C} = \{c_i:1\le i \le n; n \in \N\}$.  Set $\mathcal{C}$ represents the \textit{primary colours} and when the context is clear, referred to as primary colour-brushes or simply colour-brushes. Let the initial allocation (Step $t= 0$ of the tattooing process), say $s$, $s\ge 0$, of primary colours to a vertex $v$     be the set $X_{t=0}(v) =\{c_i: i = 1,2,3,\ldots,s; s\in \N\}$. This allocation is allowed to mutate (which is also called blending) to include all possible \textit{primary blends} of colour-brushes. Hence, this allocation is the set $\cP_0(X_{t=0}(v))$ of all non-empty subsets of $(X_{t=0}(v))$.

A \textit{primary colour blend} is the colour blend of at least two distinct primary colours. Primary colour blends may not mutate into \textit{secondary blends}. During the $i^{th}$-step of tattooing a number of identical primary colour-brushes or primary colour blends may arrive at a vertex $v$. Following from set theory repetition is not allowed. Such a set is reduced to a set of distinct primary colours or colour blends. Obviously, $\{c_1,c_2,c_3\ldots, c_s\}_{s=0} = \emptyset$. In this paper, we now write $\{c_{i,j,k,\ldots,\ell}\}$ to shortly represent $\{c_i,c_j,c_k \ldots, c_\ell\}$.

\ni Let us now introduce the notion of the tattoo power set of the set of colours as follows.

\begin{defn}\label{Defn-2.1}{\rm 
At the $i^{th}$-step, $i\ge 0$, the \textit{tattoo power set} of $X_{t=i}(v)$ is defined to be
\begin{enumerate}
	\item[(i)] $\cP^*(X_{t=i}(v)) = \cP_0(\{c_1,c_2,c_3,\ldots,c_s\})$ or $\emptyset$ (typically, but not exclusively at $0^{th}$-step), or
	\item[(ii)] $\cP^*(X_{t=i}(v)) = \cP_0(\{c_i, c_j,c_k,\ldots c_t\})$, because only primary colour-brushes have arrived, or
	\item[(iii)] $\cP^*(X_{t=i}(v)) = \cP_0(\{c_i, c_j,c_k,\ldots c_t\}\cup \{$primary colour blends $\in \cP^*_0(X_{t=i-1}(u)), \\ (u,v) \in A(G)\}$, or
	\item[(iv)] $\cP^*(X_{t=i}(v)) = \{$primary colour blends $\in \cP^*_0(X_{t=i-1}(u)), (u,v) \in A(G)\}$ only.
\end{enumerate}
}\end{defn}

\begin{exa}{\rm 
\begin{enumerate}
\item[(i)] Let $X_{t=i}(v)=\{c_1,c_2,c_3\}$. Then, we have $\cP_0(X_{t=i}(v)) = \{\{c_1\},\{c_2\},\{c_3\},\{c_1,c_2\},\{c_1,c_3\},\{c_2,c_3\}, \{c_1,c_2,c_3\}\}$. Hence, $\cP^*(X_{t=i}(v)) = \cP_0(X_{t=i}(v))$.

\item[(ii)]  Assume that during a step the colour-brushes $\{c_i, c_i, c_i,c_k,c_k, c_{s,t,\ell}\}$ arrive at vertex $v$. First, consider $\cP^*_0(\{c_i, c_i,c_i,c_k, c_k, c_{s,t,\ell}\})= \cP_0(\{c_i,c_k\}\cup \{c_{s,t,\ell}\}) = \{\{c_i\},\{c_k\},\{c_i,c_k\},\{c_{s,t,\ell}\}\}$ to decide whether sufficient colour-brushes are available to proceed tattooing. If not sufficient, consider the set $\{c_i,c_k\}\cup \{c_{s,t,\ell}\}$ and determine the minimum addition distinct primary colours with smallest subscripts needed. If for example, an additional primary colour $c_m$ is needed, consider the set $\{c_m, c_i,c_k\}\cup \{c_{s,t,\ell}\}$. Now consider the new tattoo power set $\cP^*_0(X_{t=i}(v))= \{\{c_i\},\{c_k\},\{c_m\},\{c_i,c_m\},
\{c_i,c_k\},\{c_k,c_m\},\{c_i,c_k,c_m\},\{c_{s,t,\ell}\}\}$.
\end{enumerate}
}\end{exa} 

\begin{exa}{\rm
If the vertex $v$ is initially allocated with the colour-brushes $X_{t=0}(v) = \{c_1,c_2\}$, then before tattooing begins the mapping $\{c_1,c_2\}\mapsto \cP^*_0(X_{t=0}(v)) = \{\{c_1\},\{c_2\},\{c_1,c_2\}\}$ represents the permissible mutation. Thereafter, the mapping $\{\{c_1\},\{c_2\},\{c_1,c_2\}\} \mapsto \{c_1, c_2, c_{1,2}\}$ represents the colour-brushes which may proceed with tattooing.
}\end{exa}

\begin{exa}{\rm
Assume during the $i^{th}$-step of tattooing a vertex $v$ has been allocated the colour-brushes $X_{t=i}(v) = \{c_1, c_2, c_i,c_{t,s\ldots \ell}\}$. Then before tattooing begins in the $(i+1)^{th}$-step the mapping $\{c_1, c_2, c_i,c_{t,s\ldots \ell}\}\mapsto \cP^*_0(X_{t=i+1}(v)) = \{\{c_1\},\{c_2\},\{c_i\},\{c_1,c_2\},\{c_1,c_i\}, \{c_2,c_i\},\{c_1,c_2,c_i\},\{c_{t,s\ldots \ell}\} \mapsto \{c_1,c_2,c_i,c_{1,2},c_{1,i},\\ c_{2,i},c_{1,2,i},c_{t,s\ldots \ell}\}$ represents the permissible colour-brushes.
}\end{exa}

We now introduce some more new parameters with respect to the tattooing of given graphs as follows.

\begin{defn}\label{Defn-2.2}{\rm 
The \textit{tattoo label} of an arc $a_i$, denoted $l(a_i)$, is defined to be the ordered subscript(s) of either the primary colour-brush or the blended colour-brush tattooing along the arc. Furthermore, the sum of the entries of $l(a_i)$ is denoted $l_\Sigma(a_i)$.
}\end{defn}

\begin{defn}{\rm 
The \textit{random tattoo number} of a simple connected randomly directed graph $G$ having orientation $\alpha_i(G)$,which belongs to the $2^{|E(G)|}$ possible orientations, denoted by $\tau^{\alpha_i}(G)$, is the minimum number of times primary colour-brushes are allocated to vertices of $G$ to iteratively tattoo along all arcs of $G$. The count excludes the transition of a primary colour-brush from one vertex to another.
}\end{defn}

\begin{defn}{\rm 
The \textit{tattoo number} of a simple connected graph $G$ denoted $\tau(G)$ is defined to be $\tau(G) = \min\{\tau^{\alpha_i}(G):\forall\, \alpha_i(G)\}\}$.
}\end{defn}

\begin{exa}\label{Ex-2.4}{\rm 
	Consider the directed path $P_n = v_1a_1v_2a_2v_3a_3\ldots v_{n-1}a_{n-1}v_n$, the directed cycle $C_n = v_1a_1v_2a_2v_3a_3\ldots v_{n-1}a_{n-1}v_n + (v_1,v_n)$.
\begin{enumerate}\itemsep0mm 
	\item[(i)] Let the initial colour-brush allocation of the directed path $P_n$ be $v_1 \mapsto\{c_1\},v_2\mapsto \emptyset, v_3\mapsto \emptyset \ldots v_n\mapsto \emptyset$. Clearly, iterative tattooing of all arcs is possible hence, $\tau(P_n) = 1$ with $l(a_1) = l(a_2)=\ldots = l(a_{n-1}) = (1)$.
	\item[(ii)] Let the initial colour-brush allocation of the directed cycle $C_n$ be $v_1 \mapsto\{c_1,c_2\},v_2\mapsto \emptyset, v_3\mapsto \emptyset \ldots v_n\mapsto \emptyset$. After mutation we have $v_1 \mapsto\{\{c_1\},\{c_2\},\\ \{c_1,c_2\}\},v_2\mapsto \emptyset, v_3\mapsto \emptyset \ldots v_n\mapsto \emptyset$. Clearly, tattooing is possible in six different ways.
\end{enumerate} 
Either:
\begin{enumerate}\itemsep0mm
	\item[(a)] $c_1$ tattoo along $a_1,a_2,\ldots,a_{n-1}$ and $c_2$ tattoo along $(v_1,v_n)$ with $l(a_1) = l(a_2)=\ldots = l(a_n) = (1)$ and $l((v_1,v_n)) = (2)$ or vice versa, or 
	\item[(b)] $c_1$ tattoo along $a_1,a_2,\ldots,a_{n-1}$ and $c_{1,2}$ tattoo along $(v_1,v_n)$ with $l(a_1) = l(a_2)=\ldots = l(a_n) = (1)$ and $l((v_1,v_n)) = (1,2)$  or vice versa, or
	\item[(c)] $c_2$ tattoo along $a_1,a_2,\ldots,a_{n-1}$ and $c_{1,2}$ tattoo along $(v_1,v_n)$ with $l(a_1) = l(a_2)=\ldots = l(a_n) = (2)$ and $l((v_1,v_n)) = (1,2)$  or vice versa.
\end{enumerate} 

Since $\cP^*_0(\{c_1\})= \{c_1\}$, the initial allocation of one colour-brush to a vertex will not be sufficient and therefore, we have $\tau(C_n)=2$.
}\end{exa}

\ni We now have a perhaps, obvious but important result for an undirected graph $G$.

\begin{lem}
Consider a simple connected undirected graph $G$. Over all orientations $\alpha_i$ of a graph $G$, we have $\tau(G) =\min\{\tau^{\alpha_i}(G)\} \le b_r(G)$. 
\end{lem}
\begin{proof}
As $|\cP^*_0(X_{t=i}(v))| \ge d^+(v); \forall\, v$, the proof follows immediately as an analogue  the definition of cleaning processes and the definition of the brush numbers provided in \cite{MEM2}.
\end{proof}

\ni The next result is intuitively obvious, but requires a rigorous formal proof.

\begin{thm}
Let the graph $G$ of order $n$ have an optimal orientation in respect of brush cleaning. If $k_i\ge 0$ brushes were initially allocated to vertex $v_i \in V(G)$ to realise $b_r(G)$, then at most a set $X_s(v_i)$ of $s$ colour-brushes such that $|\cP^*_0(X_s(v_i))| \ge k_i$ are required at $v$ to realise $\min\{\tau(G)\}$. 
\end{thm}
\begin{proof}
Consider a graph $G$ of order $n$ with vertices $v_i; i=1,2,3,\ldots,n$. Let the tattooing propagate on the time line $t= 0,1,2,3,\ldots, \ell$. At $t=0$, allocate to all vertices $v_i$, $d^-(v_i)=0$, $i \in \{1,2,3,\ldots, n\}$ an initial minimum set $X_{t=0}(v_i)$ of colour-brushes such that $|\cP^*_0(X_{t=0}(v_i))|$ is greater than or equal to the initial number $k_i$  of brushes allocated for brush cleaning. Clearly, $|X_{t=0}(v_i)|\le k_i$ and $|\cP^*_0(X_{t=0}(v_i))| \ge d^+(v_i)$ hence, tattooing from these vertices may propagate. Remove these vertices $v_i$ together with the out-arcs that have been tattooed. Label this reduced graph $G'$. If $E(G') = \emptyset$, tattooing is complete. Else, consider all vertices $v_j \in V(G')$, $d^-(v_j) = 0$, $d^+(v_j) > 0$. Since $\{c_\ell\} = c_\ell$ we have to consider three cases similar to those in Definition \ref{Defn-2.1}.

\textit{Case (i)}: Assume $X_{t=1}(v_j) = \{c_i,c_j,c_k,\ldots, c_s\}$, with $i,j,k, \ldots, s$ with duplication deleted. 

If $|\cP^*_0(X_{t=1}(v_j))| \ge d^+(v_j)$ no additional colour-brushes are required. So far, the tattooing count is less than or equal to the brushing count. Tattooing may proceed at $t=1$.

If $|\cP^*_0(X_{t=1}(v_j))| < d^+(v_j)$, then add the minimum additional distinct primary colours with smallest subscripts to obtain $X'_{t=1}(v_j)$ such that $|\cP^*_0(X'_{t=1}(v_j))| \ge d^+(v_j)$. Clearly, the cardinality of the set with additional colour-brushes, $|\cP^*_0(X'_{t=1}(v_j))|$ is less than the additional brushes required for brush cleaning at $t=1$.

\textit{Case (ii):} Assume $X_{t=1}(v_j) = \{c_i,c_j,c_k,\ldots, c_s\} \cup \{$primary colour blends$\}$, with $i,j,k, \ldots, s$ with duplication deleted. The further reasoning that is similar to that for Case (i).

\textit{Case (iii)} Assume $X_{t=1}(v_j)$ is the set of all primary colour blends without duplicate entries. If $|\cP^*_0(X_{t=1}(v_j))| \ge d^+(v_j)$, then no additional colour-brushes are required. Therefore, the tattooing count is less than or equal to the brushing count. If $|\cP^*_0(X_{t=1}(v_j))| < d^+(v_j)$, then add the minimum additional distinct primary colours say $\omega$, to obtain $X'_{t=1}(v_j) = \{c_1,c_2,c_3,\ldots, c_\omega \} \cup X_{t=1}(v_j)$, such that $|\cP^*_0(X'_{t=1}(v_j))| \ge d^+(v_j)$. Clearly, the cardinality of the set with additional colour-brushes, $|\cP^*_0(X'_{t=1}(v_j))|$ is less than the additional brushes required for brush cleaning at $t=1$.

The result can be settled by proceeding through cases (i) to (iii) and for $t=2,3,\ldots,\ell$.
\end{proof}

The next lemma is useful in determining the exact minimum number of additional primary colours to be added, if need be.

\begin{lem}\label{Lem-2.3}
If, in the $i^{th}$-step, a vertex $v \in V(G')$ has $d^-_{G'}(v)=0$, $d^+_{G'}(v) = \ell$ and $\kappa$ colour-brushes (after blending) allocated and $\ell > \kappa$, then 
\begin{enumerate}\itemsep0mm
\item[(i)] If no primary colours (colour-brushes) have been allocated, either the exact minimum additional distinct primary colours to be added is $\lceil log_2(\ell-\kappa+1)\rceil$,
\item[(ii)] If $\omega$ distinct primary colours have been allocated, either the exact minimum additional distinct primary colours to be added is $\lceil log_2(\frac{\ell-\kappa}{\omega+1}+1)\rceil$.
\end{enumerate}
\end{lem}
\begin{proof}
\textit{Part (i):} Since primary colour blends may not blend into secondary colours a minimum set $X$ of primary colours must be allocated in order to ensure that after mutation, at least $(\ell-\kappa)$ additional colour-brushes are available. Let $|X| =x$ and hence we solve for $x$ in the inequality $\min_x(2^x-1)\ge (\ell-\kappa)$. Therefore, a minimum of exactly $x=\lceil log_2(\ell-\kappa+1)\rceil$ primary colours must be allocated additionally.

\vspace{0.35cm}

\textit{Part (ii):} Assume that a minimum of $x$ additional primary colours are to be added. Then, after mutation, $2^x-1$ additional colour blends have been added. Further to this, the existing $\omega$ primary colours may one-on blend with these new $2^x-1$ colour blends. Hence a total of $(2^x-1)(\omega + 1)$ additional colour-brushes have been added. Hence, we must solve for $x$ in, $\min_x((2^x-1)(\omega+1)\ge (\ell-\kappa)$. Therefore, $x=\lceil log_2(\frac{\ell-\kappa}{\omega+1}+1)\rceil$. 
\end{proof}

\ni We note that first part of Lemma \ref{Lem-2.3} is a special case of the second part of Lemma \ref{Lem-2.3} and $\omega=0$ in the first part.

In \cite{TST1}, it is shown that for any tree $T$ with $d_o(T)$ vertices of odd degree, $b_r(T) = \frac{d_o(T)}{2}$. With the aforesaid as basis, the next result follows.
\begin{thm}[Erika's Theorem{\footnote{Dedicated to Erika Kok, the only sister of the first author.}}]
For any tree $T$ and any vertex $v \in V(T)$ of odd degree, there exists at least one optimal orientation corresponding to $\tau(T)$ which, after completion of tattooing, allows at least one colour-brush residing at $v$ and for a vertex $u \in V(T)$ of even degree, there does not exist any optimal orientation corresponding to $\tau(T)$ which, after completion of tattooing, allows a colour-brush residing at $u$.
\end{thm}
\begin{proof}
Consider any tree $T$ after completion of tattooing corresponding to $\tau(T)$. Add a pendent $w$ to any vertex $v$ of $T$ to obtain another tree $T'$. If $v$ has odd degree in $T$, then it has even degree in $T'$ and $d_o(T) = d_o(T')$ and hence $b_r(T) = b_r(T')$. The latter is only possible if at least one optimal orientation exists which, on completion of tattooing $T$, allows at least one colour-brush residing at $v$ .

If $v$ has even degree in $T$, then it has odd degree in $T'$ and hence $d_o(T') = d_o(T) + 2$. Therefore, $b_r(T') = b_r(T) + 1$. If an optimal orientation exists corresponding to $\tau(T)$ which, on completion of tattooing $T$, allows a colour-brush residing at $v$ then, $b_r(T') = b_r(T)$ which is a contradiction.
\end{proof}

\subsection{Tattoo Number of Certain Graphs}

Recall that a conventional \textit{friendship graph} $Fr(3,n)$, where $n\ge 1$, is a graph in which $n$ copies of the triangle $C_3$ share a common vertex.

\begin{prop}\label{Prop-2.5}
For friendship graphs $Fr(3,n);\ n\ge 1$ we have:
\begin{equation*} 
\tau(Fr(3,n)) =
\begin{cases}
2, & \text{if}\ n=1,2,\\
3, & \text{if}\ n=3,\\
i+2, & \text{if}\ i=2,3,4,\ldots\ \text{and}\ 2^i\le n \le 2^{i+1}-1.
\end{cases}
\end{equation*}
\end{prop}
\begin{proof}
\textit{Part (i):} Since $Fr(1,3)=C_3$ the result is obvious. Let the common vertex be labeled $u$ and consider the cycles per se to be labeled cycle $1$ and cycle $2$. Corresponding to the cycle number, label the vertices of cycle $1$ to be $v_{1,1}, v_{2,1}$, and those of cycle $2$, $v_{1,2}, v_{2,2}$ respectively. For tattooing to be allowed to initiate at vertex $u$ a minimum number of 3 primary colours are needed. If two primary colours are allocated to say, $v_{1,1}$ tattooing of cycle $1$ is allowed with $\{c_1,c_2\}$ or $\{c_1,c_{1,2}\}$ or $\{c_2,c_{1,2}\}$ arriving at $u$. Since $F''r(2,3) = C_3$, the result follows.

\vspace{0.25cm}

\textit{Part (ii):}  Let the common vertex be labeled $u$ and consider the cycles per se to be labeled \textit{cycle $1$, cycle $2$, cycle $3$}. Corresponding to the cycle number, label the vertices of cycle $1$, $v_{1,1}, v_{2,1}$, cycle $2$, $v_{1,2}, v_{2,2}$ and those of cycle $3$, $v_{1,3}, v_{2,3}$ respectively. For tattooing to be allowed to initiate at vertex $u$ a minimum number of $3$ primary colours are needed because $\cP_0(\{c_1,c_2,c_3\}) > 6$. If $2$ primary colours are allocated to the vertex, say $v_{1,1}$, then the tattooing of cycle $1$ is allowed with $\{c_1,c_2\}$ or $\{c_1,c_{1,2}\}$ or $\{c_2,c_{1,2}\}$ arriving at $u$. Since $F''r(3,3) = Fr(2,3)$, the result follows.

\vspace{0.25cm}

\textit{Part (iii):}  Let the common vertex be labeled $u$ and consider the cycles per se to be labeled cycle $1$, cycle $2$, cycle $3$, \ldots, cycle $n$. Corresponding to the cycle number, label the vertices of cycle $1$, $v_{1,1}, v_{2,1}$, cycle $2$, $v_{1,2}, v_{2,2}$\ldots and those of cycle $n$, $v_{1,n}, v_{2,n}$ respectively. We have $d(u)=2n$ and clearly for $n\ge 4$, it is optimal to allocated the minimum number of primary colours at vertex $u$. The proof of the latter part is also similar. Clearly, for $4 \le n\le 7$, we have $8\le d(u) \le 14$ and hence since $2^3-1< 8$ and $2^4-1 > 14$, $\tau(Fr(3,n))=4;\ 4\le n\le 7$. The latter can be expressed as $\tau(Fr(3,n))=i+2,\ i=2$ and $2^i \le n \le 2^{i+1}-1$. As all expressions are well-defined the general result follows through immediate induction.
\end{proof}

Next, we introduce the notion of a new family of graphs namely $J9$-graphs and the notion of a Joost graph as follows.  

\begin{defn}\label{Defn-2.5}{\rm 
Consider the inter-connected paths $u_1,v_{1,j},v_{2,j},\ldots,v_{n-2,j},u_2$, $n\ge 3$ for $j=1,2,3,\ldots,k$, where $k\ge 1$. The family of graphs is called $J9$-graphs.{\footnote{This graph was conceptualised in honour of Joost van der Westhuizen, former Springbok scrum-half who was diagnosed with Amyotrophic Lateral Sclerosis in 2011.}} A member of the $J9$-graphs is denoted $P^{(k)}_n$ and is called a \textit{Joost graph}.
}\end{defn}

Note that the respective end-vertices are in common. Figure \ref{fig:JoostGraph}, given below depicts the Joost graph $P^{(3)}_{7}$.

\begin{figure}[h!]
\centering
\includegraphics[width=0.9\linewidth]{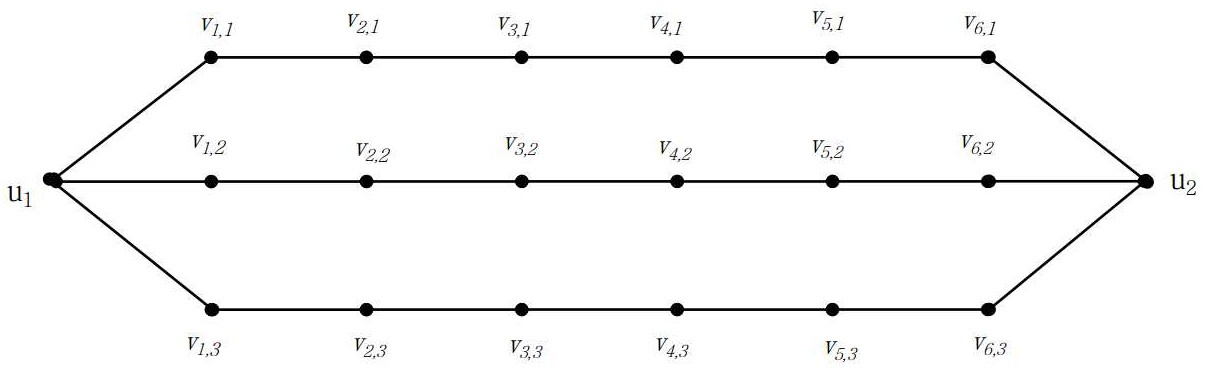}
\caption{The Joost graph, $P^{(3)}_{7}$.}
\label{fig:JoostGraph}
\end{figure}

\begin{prop}\label{Prop-2.6}
For $J9$-graphs $P^{(k)}_n$, $n\ge 3$ and $k\ge 1$, we have
\begin{equation*} 
\tau(P^{(k)}_n)) =
\begin{cases}
1, &\text {if $k=1$},\\
i+1, &\text {if $i=1,2,3,\ldots$ and $2^i\le k \le 2^{i+1}-1$}.
\end{cases}
\end{equation*}
\end{prop}
\begin{proof}
\textit{Part (i):} For $k=1,\ P^{(1)}_n=P_n$ and hence the result follows.

\textit{Part (ii):} For $k=2$, $P^{(2)}_n = C_n$ hence the result. For $k=3$ and considering Definition \ref{Defn-2.5} the common end-vertices are $u_1,u_2$. Consider the paths to be path $1$, path $2$, path $3$, respectively. From Definition \ref{Defn-2.5} it follows that the internal vertices are $v_{1,1},v_{2,1},\ldots, v_{n-2,1}$, $v_{1,2},v_{2,2},\ldots, v_{n-2,2}$ and $v_{1,3},v_{2,3},\ldots, v_{n-2,3}$. As $d(u_1)=d(u_2)=3$ and $2^2-1=3$ the allocation of $\{c_1,c_2\}$ to the vertex, say $u_1$, suffices and is clearly a minimum allocation. Combining $k=2,3$ allows the expression $\tau(P^{(k)}_n)=i+1,\ i=1$, $2^1 \le k \le 2^{1+1}-1$. Since all expressions are well-defined, the general result follows through immediate induction.
\end{proof}

Note that the result in Proposition \ref{Prop-2.6} is equivalent to $\tau(P^{(k)}_n)) = i+1$, if $i= 0,1,2,\ldots$ and $2^i\le k \le 2^{i+1} -1$. We also observe that the method of proof of Proposition \ref{Prop-2.5} and Proposition \ref{Prop-2.6} illustrate the principles that in a friendship graph ($Fr(3,n)$) the optimal tattooing sequence must initiate at a vertex with minimum degree, $\delta(Fr(3,n))$ whilst for a Joost graph ($P^{(k)}_n$), the optimal tattooing sequence must begin at a vertex with maximum degree, $\Delta(P^{(k)}_n)$.

\section{Conclusion}

In this paper, we have discussed a new concept namely tattooing, as a colouring extension of cleaning models used in graph theory and introduced some interesting concepts and parameters in that area. We have also pointed out that the concept of tattooing has many real world applications and many of them are still to be explored and to be discovered. This impression is mainly based on the fact that many biological or virtual propagation models or mechanisms rely on mutation to reach a threshold level before propagation ignites. 

Besides determining the invariants $b_\tau(G), \mathfrak{T}_{FSG}(G)$ and $\tau(G)$ the invariant $\mathfrak{T}(G)$ is open for complexity and probability analysis. Determining the tattoo number of different graph classes offers much for further investigations. An algorithmic study of this graph parameter is also worth for future studies.

\end{document}